\documentclass[12pt]{article}
\usepackage[T1]{fontenc} 
\usepackage{amsfonts}    
\usepackage{amsmath}     
\usepackage{amssymb}     
\usepackage{graphicx}    
\usepackage{amsthm}      
\usepackage{geometry}    

\newtheorem{theorem}{Theorem}[section]
\newtheorem{lemma}[theorem]{Lemma}
\newtheorem{proposition}[theorem]{Proposition}
\newtheorem{corollary}[theorem]{Corollary}
\newtheorem{definition}[theorem]{Definition}

\newtheorem{Claim}{Claim}

\usepackage{xcolor}

\newcommand{\N}{\mathbb N}

\newcommand{\Z}{\mathbb Z}
\newcommand{\ov}{\overline}

\title{On the Calculus of the spectrum of the Laplacian for functions in Lens Spaces}
\author{Gedeana Pantoja da Silva, Alexandre Casassola Gon\c{c}alves, Ferreira da Silva Rafael}
\date{\today} 

\begin{document}
	\maketitle
	
	\begin{abstract}

  We improve a specific method to obtain the dimension of the eigenspaces of the Laplace-Beltrami operator on lens spaces and establish some applications  related to the explicit description of the dimension of the smallest positive eigenvalue    and  the parity of the dimensions of the eigenspaces.

\end{abstract}

\section{Introduction}

  The study of the spectrum of the Laplacian on closed Riemannian manifolds has proven to be extremely useful both in mathematics and in physics, with applications in areas like the study of  wave equations, explicit computation of traces of operators,  calculus of the regularized determinant  and  of the analytic torsion.
 One of the main features of this operator that makes it useful in applications is the fact that its spectrum is discrete and has no accumulation points, which makes it possible to define and study several related concepts, especially the ones connected to spectral zeta functions.
 
  For several of the applications of the Laplacian it is necessary, or at least useful, to  utilize  a precise description of the spectrum and dimension of the eigenspaces of this operator and this kind of description has been the subject of research in mathematics for decades in works like \cite{beitz2020spectrum}, \cite[Chapter III]{berger1971spectre}, \cite{boucetta1999spectre}, \cite{ikeda1978spectra}, \cite{iwasaki1979spectra}, \cite{lauret2016spectra}, \cite{lauret2018spectrum}, \cite{levy2006spectre}, \cite{pilch1984formulas},  \cite{prufer1985spectrum2}, \cite{prufer1989spectrum}, \cite{Puta}, \cite{rafael2022determinantodd}, \cite{rafael2020regularized},  \cite{rafael2018determinantes-projetivos-pares}, \cite{sakai1976spectrum}, \cite{tsagas1983spectrum}, \cite{weng1996analytic} and \cite{yamamoto1980number}.   
   Among the several articles that studied the spectrum and dimension of the eigenspaces of the Laplacian, one of the most prominent is \cite{lauret2016spectra}, where the authors established a description of the multiplicities of the eigenvalues of the Laplacian  acting on sections of   natural vector bundles over spherical space forms.
 
 The procedure developed in \cite{lauret2016spectra}  may be roughly summarized  in the following way: the authors considered the traditional definition of spherical space forms as quotients $\Gamma \backslash G/K$, where $G=SO(2m)$, $K\simeq SO(2m-1)$ and $\Gamma$ is a discrete subgroup of $SO(2m)$ acting freely on $G$,  they considered unitary representations $(\tau, W_{\tau})$ of $K$, they used these spaces and representations to construct homogeneous vector bundles   $E_{\tau}:= G \times _{\tau} W_{\tau}$, so that the spaces of smooth sections $\Gamma ^{\infty}(E_{\tau})$ and  $C^{\infty}(G/K;\tau):=\{ f:G\rightarrow W_{\tau} \; | \; f(xk)=\tau (k^{-1})f(x)\}$ are isomorphic, then they used the complexification $g$ of the Lie algebra $g_0$ of $G$ and the Casimir element $C$ of $g$ to obtain the second order elliptic differential operators $\Delta _{\tau}$ on $C^{\infty}(G/K;\tau)$ and $\Delta _{\tau, \Gamma}$ on $\Gamma ^{\infty}(\Gamma \backslash E_{\tau})$. After these basic constructions, the authors developed several results that described the behaviour of the spectrum of  $\Delta _{\tau, \Gamma}$ and allowed them to achieve their main objective, which was finding examples of Riemannian manifolds whose Laplacians are isospectral on forms of arbitrary degree, but not strongly isospectral.
  
  In the specific case of the Laplace-Beltrami operator acting on functions over lens spaces $L=L(p;q_1,\dots , q_m)$, denoting by  $\mathcal{L}=\mathcal{L}(p;q_1,\dots ,q_m)$ the congruence lattice associated to $L$, the multiplicities of the eigenvalues $i(i+n-1)$ of the Laplacian $\Delta ^0_{L}$ in terms of the numbers $N_{\mathcal{L}}(h)$ of congruence lattice points with $||\cdot ||_{1}$-length equal to $h$ are given on \cite[page 1055]{lauret2016spectra} by the expression
   \begin{equation} \label{dim-auto-esp}
   \sum\limits _{s=0}^{\lfloor i/2\rfloor} \left( \begin{array}{c}  s+m-2 \\ m-2  \end{array}  \right)   N_{\mathcal{L}}(i-2s),
   \end{equation}
  but the authors did not develop an efficient method to calculate the numbers $N_{\mathcal{L}}(h)$, since this was not an objective of their research. In order to effectively calculate the multiplicity of the eigenvalue $i(i+n-1)$ of $\Delta ^0_{L}$ using exclusively the theory of \cite{lauret2016spectra} we would need to calculate first the numbers  $N_{\mathcal{L}}(h,l)$  of congruence lattice points with $||\cdot ||_{1}$-length equal to $h$ and with exactly $l$ entries equal to zero, and then calculate 
  $$N_{\mathcal{L}}(h)=\sum\limits _{l=0}^mN_{\mathcal{L}}(h,l).$$
  
  The expresion for $N_{\mathcal{L}}(h,l)$ is given in \cite[Theorem 4.2]{lauret2016spectra} by
  \begin{equation} \label{Lauret NL(h,l)}
  N_{\mathcal{L}}(h,l)=
  \sum\limits _{s=0}^{m-l} 
  2^s\left( \begin{array}{c} l+s \\ s  \end{array} \right)
   \sum\limits _{t=s}^{\alpha} \left( \begin{array}{c} t-s+m-l-1 \\ m-l-1  \end{array} \right)
   N_{\mathcal{L}}^{\mathrm{red}}(h-tq,l+s),
  \end{equation}
  where $N_{\mathcal{L}}^{\mathrm{red}}(h-tq,l+s)$ is described in \cite[Definition 4.1]{lauret2016spectra}.

 The main result of the present text is Theorem \ref{formula} where we establish an expression for   $N_{\mathcal{L}}(h)$ that does not depend on the numbers $N_{\mathcal{L}}(h,l)$ and therefore is smaller and easier to use. As applications we prove properties of the spectrum of the Laplace-Beltrami operator on lens spaces, specially properties related to the dimension of the eigenspaces.
 
  The formula obtained by us for $N_{\mathcal{L}}(k+np)$ with $0\leq k <p$ and $n\in \N \cup \{0\}$ is given by:
 \begin{equation*}
		N_{\mathcal{L}}(k+np) =\sum_{t=0}^{m-1}\sum_{U\subseteq M}C_{m-1}^{n-t+|U|-1}\gamma(U,k+tp),
	\end{equation*}
	where $C_{a}^{b}$ is the binomial coefficient  $\frac{b!}{a!(b-a)!}$ and   $ \gamma(U,k+tp) $, which are described precisely in Definition \ref{NL(k)} below, are the analogous of the numbers $ N_{\mathcal{L}}^{\mathrm{red}}(h-tq,l+s)$ in expression (\ref{Lauret NL(h,l)}).

\section{Preliminaries and Notation}
Let $p,q_1,q_2,\dots,q_m$ be natural numbers such that for every $i$, $q_i$ is co-prime to $p$, and let $L(p;q_1,\dots,q_m)$ denote the corresponding  lens space.
We will use $M$ to denote the set $M=\{1,2,\dots,m\}$, $U$ to denote an arbitrary subset of $M$ and we will denote by $u:=|U|$ the cardinality of $U$.  The \textit{congruence lattice} associated to $U$  is the set
\begin{equation}
\label{lattice}	\mathcal{L}(U)=\{(x_{i_1},x_{i_2},\dots,x_{i_u})\in\mathbb{Z}^u : x_{i_1}q_{i_1}+x_{i_2}q_{i_2}+\cdots +
	x_{i_u}q_{i_u}\equiv0(\text{ mod }p)\}
\end{equation}
and for any $x=(x_{i_1},\dots,x_{i_u})\in \mathcal{L}(U)$, the 1-norm of $x$ is given by $\|x\|_1=|x_{i_1}|+\cdots+|x_{i_u}|$. 

We may use the $1$-norm to describe special subsets of the congruence lattice.
	Given $U\subset M$, and $k\geq0$, we set:
	\begin{align}
	\nonumber	\Omega(U,k)&=\{x\in\mathcal{L}(U);\ \|x\|_1=k\},\\
		C(U,k)&=\{x\in\Omega(U,k);\;   \forall\ i\in U,\, |x_i|<p\ \}.\label{L,Omega,C}
	\end{align}

We are interested in certain coefficients related to the congruence lattices. 
Those are defined as follows:
\begin{definition}\label{NL(k)} 
	For $k\geq0$ and $U\subset M$ we define:   
	\begin{equation}
	  N_{\mathcal{L}}(U,k)=|\Omega(U,k)|.
        \end{equation}
        When $U=M$ we write, for short, $N_{\mathcal{L}}(k):=N_{\mathcal{L}}(M,k)$.
        We also set
        \begin{equation}
          	\gamma(U,k)=|C(U,k)|.
	\end{equation}
\end{definition}

Notice that our definition of $N_{\mathcal{L}}(k)$ coincides with the definition present in \cite[expression $(11)$]{lauret2016spectra}, except for the fact that in \cite[expression $(11)$]{lauret2016spectra} the author allowed $\mathcal{L}$ to be any sublattice of $\Z ^m$, while  we are specializing in lattices as described in expression (\ref{lattice}). On the other hand, our definition of $N_{\mathcal{L}}(U,k)$ diverges from the definition of $N_{\mathcal{L}}(k,l)$ present in \cite[expression $(12)$]{lauret2016spectra} and the difference between $N_{\mathcal{L}}(U,k)$ from our text and $N_{\mathcal{L}}(k,l)$ from \cite[expression $(12)$]{lauret2016spectra} is the foundation of the differences between our study and the one developed in \cite{lauret2016spectra}.

Our main objective is to compute 
$N_{\mathcal{L}}(k)=N_{\mathcal{L}}(M,k)$, and the sets defined below will assist us in  this objective. 
 For integers $n\geq0$, $k\in\{0,1,\dots,p-1\}$ and $N\subset M$, we define:
	\begin{equation}\begin{aligned}
	    \Omega_N(k+np)&=\Omega_N(M,k+np)\\
            &:=\{x\in\Omega(M,k+np);\; ( x_i<0\text{ and }x_i\equiv 0\text{ mod }p )
		\iff i\in N\}.
	\end{aligned}\end{equation}
Thus, $\Omega_N(k+np)$ gathers the solutions of the linear congruence that have norm 1 equal to $k+np$
and have the coefficient $x_i$ as a negative multiple of  $p$ if and only if $i$ belongs to $N$.

The following claims involve the sets $C(U,k+tp)$, $\Omega_N(k+np)$ and $\Omega(M,k+np)$, they compose the basis for the main results of the text.
\begin{Claim} If $U\subset M$, and $t\ne t'$, then $C(U,k+tp)\cap C(U,k+t'p)=\emptyset$.
\end{Claim}
\begin{proof}
	This is a direct consequence of the definition of $\Omega$.
\end{proof}
\begin{Claim}\label{afirm-1}
	If $N,N'\subset M$ with $N\ne N'$, then $\Omega_N\cap \Omega_{N'}=\emptyset$.
\end{Claim}
\begin{proof}
  Assume, for the sake of contradiction, that $N\ne N'$ and $(x_1,\dots,x_m)\in \Omega_N\cap \Omega_{N'}$.
  Without loss of generality, suppose $j\in N$ with $j\notin N'$. By the definition of $\Omega_N$, the coordinate $x_j$
  is a negative multiple of $p$.

  On the other hand, since $j\notin N'$,  the coordinate $x_j$ is not a negative multiple of $p$, which is a contradiction.
  Therefore, there is no $(x_1,\dots,x_m)\in \Omega_N\cap \Omega_{N'}$, that is, $\Omega_N\cap\Omega_{N'}=\emptyset$.
\end{proof}

\begin{Claim}\label{afirm-2}
	Every element $(x_1,\dots,x_m)\in\Omega(M,k+np)$ belongs to $\Omega_N(k+np)$ for some $N\subset M$.
\end{Claim}
\begin{proof}
	Indeed, given $(x_1,\dots,x_m)\in\Omega(M,k+np)$,  it suffices to consider $N=\{i\in M; x_i<0 \mbox{ and }x_i
	\equiv0(\text{ mod }p)\}$. Thus $(x_1,\dots,x_m)\in\Omega_N(k+np)$.
\end{proof}

The previous statements lead us to the following proposition:
\begin{proposition}\label{prop7} For every $k\in \{ 0,1,\dots ,p-1  \}$ and $n \in \N \cup \{0\}$ we have:
	\begin{equation}
		\Omega(M,k+np)=\bigsqcup_{N\subseteq M}\Omega_N(k+np).
	\end{equation}
\end{proposition}
\begin{proof}
By definition, $\Omega_N(k+np)\subseteq\Omega(M,k+np)$. From {Claim \ref{afirm-2}},
we have $\Omega(M,k+np)\subseteq\cup_{N\subset M}\Omega_N(k+np)$ and by {Claim \ref{afirm-1}}, this union must be disjoint.
\end{proof}


\section{Main Results}

The strategy for computing $N_{\mathcal{L}}(k+np)$ is to establish a function $f$ that maps $\Omega(M,k+np)$,
which is the disjoint union of the subsets $\Omega_N(k+np)$, for $N\subseteq M$, {onto a subset of the  
  union of the} sets $C(U,k+tp)$, for $U\subset M$ and $t\geq 0$.

To be more specific, we define for each $N\subseteq M$ a function $f_N:\Omega_N(M,k+np)\to \bigsqcup_{t\geq0}C(\ov N,{k+tp})$,
given by $f_N(x_1,\dots,x_m)=(\overline{x}_1,\dots,\hat{x}_{i_1},\dots,\hat{x}_{i_{|N|}},\dots,\overline{x}_m)$. The indices
of the suppressed coefficients are precisely $\{i_1,i_2,\dots,i_{|N|}\}=N$. We use the notation $\ov N$ to denote
the complement $\ov N=M-N$.

In this expression we have the following convention for $\ov{x_i}$:
\begin{equation}\label{def-xibar}
	\ov{x_i}=\begin{cases}
		x_i\mod p, &\mbox{if }x_i\ge 0\\
		(x_i\mod p)-p, &\mbox{if }x_i<0,
	\end{cases}
\end{equation}
where $x_i\ \text{mod}\ p$ is the representative of the congruence class modulo $p$ of $x_i$ in $\{0,1,2,\dots,p-1\}$.

Given the functions $f_N$, we define
\begin{equation}
	f:\Omega(M,k+np)=\bigsqcup_{N\subseteq M}\Omega_N(M,k+np)\longrightarrow \bigsqcup_{U\subseteq M, t\geq 0} C(U,{k+tp}),
\end{equation}
by
\begin{equation}
	f(x_1,\dots,x_m)=f_N(x_1,\dots,x_m)\quad\text{if}\ (x_1,\dots,x_m)\in\Omega_N(k+np).
\end{equation}

\begin{lemma}\label{main-th} The function $f$ is well defined. For each $N\subset M$, 
$f|_{\Omega _N(k+np)}$ is surjective onto $\bigsqcup _{t=0}^{n-|N|} C(\ov N,{k+tp})$. Given $y\in C(\ov{N},{k+tp})$, the cardinality of the preimage of $y$ by $f$ is
\begin{equation}
  |f^{-1}(\{y\})|=C^{n-t+|\ov N|-1}_{m-1}.
\end{equation}
\end{lemma}
\begin{proof}
To show that $f$ is well defined it suffices to prove that for $N\subset M$ the functions $f_N$ are well defined.
Fixing $N$, and taking $x=(x_1,\dots,x_m)\in\Omega_N(M,k+np)$, let $y=f_N(x_1,\dots,x_m)$, each $\ov x_i$ 
given by convention \eqref{def-xibar}.
Then $y$ has $|\ov N|$ coefficients, since only the coefficients of the indices $i\in N$ were removed from
$x=(x_1,\dots,x_m)$.
Furthermore, $|y_j|=|\ov x_j|<p$ for $j\in\ov N$, what comes straight from convention \eqref{def-xibar}.
It only remains to show that $y$ in fact satisfies the linear congruence condition, and that its 1-norm is $k+tp$ for
some $t\geq0$.

So we write $\ov N=\{j_1,\dots,j_{|\ov N|}\}$ and $y=(y_{j_1},\dots,y_{j_{|\ov N|}})$, we must check that 
\begin{equation}\label{cong-id-y}
  q_{j_1}y_{j_1} + \dots + q_{j_{|\ov N|}}y_{j_{|\ov N|}}\equiv 0\ (\mbox{mod } p).
\end{equation}

  Since $x=(x_1,\dots, x_m)\in\Omega_N (M,k+np)$, we have:
  \begin{equation}
    q_1{x}_1+\dots+q_m{x}_m \equiv 0\ (\mbox{mod } p),
  \end{equation}
  that is, $\sum_{i\in M}q_ix_i$ is a multiple of $p$. On the other hand,
  \begin{equation}
    \sum_{i\in M}q_ix_i=\sum_{i \in N} q_i x_i + \sum_{j \in \ov{N}} q_j x_j,
  \end{equation}
  and since the first sum of the second member of the above identity is made up of multiples of $p$,
  the second sum of the second member must also be a multiple of $p$. Thus
  \begin{equation}
    \sum_{j \in \ov{N}} q_j x_j \equiv 0\ (\mbox{mod } p).
  \end{equation}
  For each $j\in\ov N$, $x_j\equiv \ov x_j (\,\text{mod }p)$, so
  \begin{equation}
    \sum_{j \in \ov{N}} q_j y_j=\sum_{j \in \ov{N}} q_j \ov x_j\equiv \sum_{j \in \ov{N}} q_jx_j\equiv 0(\,\text{mod }p).
  \end{equation}
  Thus, we have that $y$ satisfies \eqref{cong-id-y}, hence belongs to $C(\ov N, h)$ for some $h=\|y\|_1\geq0$.
  Finally, let us show that $h=k+tp$ with $t\leq n-|N|$.
		
  For each $j\in\ov N$ there exists $\alpha_j \in \Z$ such that $y_j=x_j-\alpha_jp$. convention \eqref{def-xibar}
  gives us $|x_j-\alpha_jp|=|x_j|-|\alpha_j|p$. Thus, 
  \begin{equation}\begin{aligned}\label{norma-y-x}
      \|y\|_1&=\sum_{j\in\ov N}|y_j|=\sum_{j\in\ov N}|x_j-\alpha_jp|=\sum_{j\in\ov N}|x_j|-p\sum_{j\in\ov N}|\alpha_j|=\\
      &=\sum_{j\in\ov N}|x_j|+\sum_{i\in N}|x_i| -p\sum_{j\in\ov N}|\alpha_j|-\sum_{i\in N}|x_i|=\\
      &=\|x\|_1-p\left(\sum_{j\in\ov N}|\alpha_j|+\sum_{i\in N}\frac{|x_i|}p\right).
  \end{aligned}\end{equation}
  
  Recall that $N$ is the set of indices $i$ such that $x_i$ is a negative multiple of $p$, then $|x_i|/p$ is an
  integer greater than or equal to 1. The quantity in parentheses in the last member of \eqref{norma-y-x} is
  an integer $s$ that satisfies $s\geq |N|$. That is,
  \begin{equation}
    \|y\|_1=k+np-sp=k+(n-s)p=k+tp,
  \end{equation}
  with $t\leq n-|N|$. Therefore, we can state that $f_N$ is well defined, and so is $f$.
	
  Now suppose that for some $N \subset M$ and $0\leq t\leq n-|N|$, it holds $C(\ov N,k+tp)\neq\emptyset$.
  Let $y = (y_{j_1}, y_{j_2}, \dots,y_{j_{\ov N,}}) \in C(\ov {N},k+tp)$. We will find $x \in \Omega(M,k+np)$ such
  that $f(x)=y$. Without loss of generality, we assume $N =\{1,2, \dots, |N|\}$, and therefore
  $\ov N=\{|N|+1,|N|+2,\dots,m\}$. For each $i \in \{ 1, \dots , m\}$ we set $\tilde x_i = -p$ if $i \in N$
  and $\tilde x_i = y_i $ if $i\in \ov N$, resulting in:
  \begin{equation*}
    \tilde{x} = (-p,-p, \dots, -p, y_{|N|+1},y_{|N|+2},\dots, y_m).
  \end{equation*}
  
  Of course, $\tilde x$ satisfies the linear congruence, that is, $\tilde x\in\Omega(M,\|\tilde x\|_1)$. At this stage
  we have $\| \tilde{x} \|_1 = \|y\|_1 + |N|p=k+tp+|N|p$.
  	
  To arrive at a point $x$ with $\|x\|_1=k+np$ we must add the difference of the norms $\|x\|_1-\|\tilde x\|_1=(n-t-|N|)p$
  to the coordinates of $\tilde x$ so that each coordinate preserves its sign
  and is increased by multiples of $p$. Note that this is only possible if $n-t-|N|\geq0$, and this is where we use the
  $t\leq n-|N|$ assumption. In particular, we can add these multiples to the first coordinate, and put
  \begin{equation}
    x = (-(n-t-|N|+1)p,-p, \dots, -p,y_{|N|+1},y_{|N|+2},\dots, y_m).
  \end{equation}
  
    The last statement of the theorem is proven by means of a standard combinatorial argument.
    Given a non-negative integer $a$, the number of ways to write $a$ as the sum of $m$ ordered non-negative integers
    is $C_{m-1}^{a+m-1}=\frac{(a+m-1)!}{(m-1)!a!}$.
    
    To construct $x$, after having defined $\tilde x$, we must distribute $(n-t-|N|)$ ``packets'' of $p$ units
    in $m$ coordinates, in an arbitrary way. This problem is then the number of ways of summing $m$ non-negative
    integers adding up to $n-t-|N|$, that is, $C_{m-1}^{n-t-|N|+m-1}$. Since $|\ov N|=m-|N|$ we arrive at 
    \begin{equation}
      |f^{-1}(\{y\})|=C^{n-t-|N|+m-1}_{m-1}=C^{n-t+|\ov N|-1}_{m-1}.
    \end{equation}
    This completes the proof.
\end{proof}

\begin{theorem}\label{formula} For $k\in\{0,1,\dots,p-1\}$ and $n\in\N$ we have
	\begin{equation}\label{NLknp}
		N_{\mathcal{L}}(k+np) =\sum_{t=0}^{m-1}\sum_{U\subseteq M}C_{m-1}^{n-t+|U|-1}\gamma(U,k+tp),
	\end{equation}
\end{theorem}
\begin{proof}
	Let $N\subset M$ and $0\leq t\leq n-|N|$ be fixed. In view of Theorem \ref{main-th} we have that, for each $y\in C(\overline{N},k+tp)$,
	$y$ has $C_{m-1}^{n-t+|\ov N|-1}$ preimages in $\Omega_N(M,k+np)$. Therefore,
	\begin{equation}
		|f^{-1}(C(\ov N,k+tp))|=C_{m-1}^{n-t+|\ov N|-1}.|C(\ov N,k+tp))|=C_{m-1}^{n-t+|\ov N|-1}.\gamma(\ov N,k+tp),
	\end{equation}

	We recall that by a property of the function $\Gamma$ (factorial), the combinatorics holds $C^a_b=0$ whenever
	$b>a\geq0$. Therefore, there is no problem in considering for counting purposes, $t>n-|N|$, in which case $f^{-1}(\{y\})$
	will be empty. And for $t\geq m$ it is immediate that $\gamma(\ov N,k+tp)=0$, since in this case $C(\ov N,k+tp)$ is empty.
	Therefore, the image of $\Omega_N(M,k+np)$ by $f$ is contained in $\sqcup_{0\leq t\leq m-1}C(\ov N,k+tp)$, and therefore
	\begin{equation}
		|\Omega_N(M,k+np)|=\sum_{t=0}^{m-1}C_{m-1}^{n-t+|\ov N|-1}.\gamma(\ov N,k+tp).
	\end{equation}
	And since $\Omega(M,k+np)=\bigsqcup_{N\subseteq M}\Omega_N(M,k+np)$ (Proposition \ref{prop7}) it is immediate that
	\begin{equation}\begin{aligned}
			N_{\mathcal{L}}(k+np) &=|\Omega(M,k+np)|=\sum_{N\subseteq M}|\Omega_N(M,k+np)|=\\
			&=\sum_{t=0}^{m-1}\sum_{U\subseteq M}C_{m-1}^{n-t+|U|-1}\gamma(U,k+tp).  
	\end{aligned}\end{equation}
\end{proof}

\section{Applications}

\begin{corollary} Let $p\in\N$ be an even integer. Then for any $i\in\N$ odd, $\text{dim}(\lambda_i)$ is even.
\end{corollary}
\begin{proof}
  If $s$ is a positive integer then, for every $x\in C(U,s)$ we have $-x\neq x$ and $-x\in C(U,s)$. Thus $\gamma(U,s)$
  is even.

  Now, for $s\geq1$ an odd integer, since $p$ is even, we can write $s=k+np$ with $k$ an odd number in $\{0,1,\dots,p-1\}$
  and $n\geq0$. In particular, the term $k+tp$ occurring in the second hand side of \eqref{NLknp} is odd, for all $t\geq0$.
  Hence $k+tp$ will never be zero, and $N_{\mathcal{L}}(s)$ will be even.

  Finally, back to identity \eqref{dim-auto-esp}, the number $i-2s$ is odd, and so $N_{\mathcal{L}}(i-2s)$
  is even. Therefore $\text{dim}(\lambda_i)$ is even.
\end{proof}

	
\par\normalfont\fontsize{12}{12}\selectfont
\bibliographystyle{abbrv}
\bibliography{ref-artigo.gedeana}

\begin{thebibliography}{10}

\bibitem{beitz2020spectrum}
S.~F. Beitz et~al.
\newblock Spectrum of generalized {H}odge-{L}aplace operators on flat tori and
  round spheres.
\newblock {\em Osaka Journal of Mathematics}, 57(2):357--379, 2020.

\bibitem{berger1971spectre}
M.~Berger, P.~Gauduchon, E.~Mazet, M.~Berger, P.~Gauduchon, and E.~Mazet.
\newblock {\em Le spectre d'une vari{\'e}t{\'e} riemannienne}.
\newblock Springer, 1971.

\bibitem{boucetta1999spectre}
M.~Boucetta.
\newblock Spectre des laplaciens de {L}ichnerowicz sur les sph{\`e}res et les
  projectifs r{\'e}els.
\newblock {\em Publicacions Matem{\`a}tiques}, 43(2):451--483, 1999.

\bibitem{Puta}
M.~Craioveanu, M.~Puta, and T.~Rassias.
\newblock {\em Old and new aspects in spectral geometry}, volume 534.
\newblock Springer Science \& Business Media, 2013.

\bibitem{ikeda1978spectra}
A.~Ikeda and Y.~Taniguchi.
\newblock Spectra and eigenforms of the {L}aplacian on {$S^n$} and {$P^n (C)$}.
\newblock {\em Osaka J. Math}, 15(3):515--546, 1978.

\bibitem{iwasaki1979spectra}
I.~Iwasaki and K.~Katase.
\newblock On the spectra of laplace operator on ${\Lambda}^{\ast}({S}^n)$.
\newblock {\em Proc. Japan Acad. Ser. A Math. Sci.}, 55(4):141--145, 1979.

\bibitem{lauret2018spectrum}
E.~A. Lauret.
\newblock The spectrum on p-forms of a lens space.
\newblock {\em Geometriae Dedicata}, 197:107--122, 2018.

\bibitem{lauret2016spectra}
E.~A. Lauret, R.~J. Miatello, and J.~P. Rossetti.
\newblock Spectra of lens spaces from 1-norm spectra of congruence lattices.
\newblock {\em International Mathematics Research Notices}, 2016(4):1054--1089,
  2016.

\bibitem{levy2006spectre}
A.~L{\'e}vy-Bruhl-Laperri{\`e}re.
\newblock Spectre du de {R}ham-{H}odge sur l'espace projectif quaternionique.
\newblock In {\em S{\'e}minaire d'Alg{\`e}bre Paul Dubreil et Marie-Paule
  Malliavin: Proceedings, Paris 1980 (33{\`e}me Ann{\'e}e)}, pages 98--129.
  Springer, 2006.

\bibitem{pilch1984formulas}
K.~Pilch and A.~N. Schellekens.
\newblock Formulas for the eigenvalues of the {L}aplacian on tensor harmonics
  on symmetric coset spaces.
\newblock {\em Journal of mathematical physics}, 25(12):3455--3459, 1984.

\bibitem{prufer1985spectrum2}
F.~Pr{\"u}fer.
\newblock On the spectrum and the geometry of a spherical space form {II}.
\newblock {\em Annals of Global Analysis and Geometry}, 3:289--312, 1985.

\bibitem{prufer1989spectrum}
F.~Pr{\"u}fer.
\newblock The spectrum and quadrature formulas of spherical space forms.
\newblock {\em Zeitschrift f{\"u}r Analysis und ihre Anwendungen},
  8(2):121--130, 1989.

\bibitem{rafael2020regularized}
F.~S. Rafael.
\newblock Regularized determinant of the {L}aplacian on forms over spheres.
\newblock {\em S{\~a}o Paulo Journal of Mathematical Sciences}, 14(2):539--561,
  2020.

\bibitem{rafael2018determinantes-projetivos-pares}
F.~S. Rafael.
\newblock Determinants and {C}heeger-{M}{\"u}ller {T}heorem on even dimensional
  projective spaces.
\newblock {\em SciELO Preprints}, 2021.

\bibitem{rafael2022determinantodd}
F.~S. Rafael.
\newblock Regularized determinant of the {L}aplacian on forms over odd
  dimensional projective spaces.
\newblock {\em SciELO Preprints}, 2022.

\bibitem{sakai1976spectrum}
T.~Sakai.
\newblock On the spectrum of lens spaces.
\newblock In {\em Kodai Mathematical Seminar Reports}, volume~27, pages
  249--256. Department of Mathematics, Tokyo Institute of Technology, 1976.

\bibitem{tsagas1983spectrum}
G.~Tsagas.
\newblock The spectrum of the {L}aplace operator for a spherical space form.
\newblock {\em Atti della Accademia Nazionale dei Lincei. Classe di Scienze
  Fisiche, Matematiche e Naturali. Rendiconti Lincei. Matematica e
  Applicazioni}, 74(6):357--365, 1983.

\bibitem{weng1996analytic}
L.~Weng and Y.~You.
\newblock Analytic torsions of spheres.
\newblock {\em International Journal of Mathematics}, 7(1):109--126, 1996.

\bibitem{yamamoto1980number}
Y.~Yamamoto.
\newblock On the number of lattice points in the square $x+ y \leq u$ with a
  certain congruence condition.
\newblock {\em Osaka J. Math}, 17(1):9--21, 1980.

\end{thebibliography}

\end{document}